\documentclass{article}

\usepackage{arxiv}

\usepackage[utf8]{inputenc} 
\usepackage[T1]{fontenc}    
\usepackage{hyperref}       
\usepackage{url}            
\usepackage{booktabs}       
\usepackage{amsfonts}       
\usepackage{nicefrac}       
\usepackage{microtype}      
\usepackage{lipsum}
\usepackage{amsmath}
\usepackage{euscript}
\usepackage{amssymb}
\usepackage{amsthm}
\usepackage{amsopn}
\usepackage{setspace}
\usepackage{ dsfont }
\usepackage{float}
\usepackage{xcolor}

\usepackage[normalem]{ulem}

\usepackage{pgfplots, tikz}
\pgfplotsset{
    midpoint segments/.code={\pgfmathsetmacro\midpointsegments{#1}},
    midpoint segments=3,
    midpoint/.style args={#1:#2}{
        ybar interval,
        domain=#1+((#2-#1)/\midpointsegments)/2:#2+((#2-#1)/\midpointsegments)/2,
        samples=\midpointsegments+1,
        x filter/.code=\pgfmathparse{\pgfmathresult-((#2-#1)/\midpointsegments)/2}
    }
}

\pgfplotsset{
    right segments/.code={\pgfmathsetmacro\rightsegments{#1}},  
    right segments=3,
    right/.style args={#1:#2}{
        ybar interval,
        domain=#1+((#2-#1)/\rightsegments):#2+((#2-#1)/\rightsegments),
        samples=\rightsegments+1,
        x filter/.code=\pgfmathparse{\pgfmathresult-((#2-#1)/\rightsegments)}
    }
}

\pgfplotsset{
    left segments/.code={\pgfmathsetmacro\leftsegments{#1}},
    left segments=3,
    left/.style args={#1:#2}{
        ybar interval,
        domain=#1:#2,
        samples=\leftsegments+1,
        x filter/.code=\pgfmathparse{\pgfmathresult}
    }
}

\pgfplotsset{compat=1.14}

\usepackage{tikz-cd}
\usetikzlibrary{calc, shapes, backgrounds,
shapes.geometric, arrows, mindmap,
positioning, fadings, through,
arrows, decorations.markings, patterns}

\newcommand{\R}{\mathbb{R}}

\newcommand{\bp}{\begin{proof}}
\newcommand{\ep}{\end{proof}}
\newcommand{\beas}{\begin{eqnarray*}}
\newcommand{\eeas}{\end{eqnarray*}}

\newtheorem{thm}{Theorem}[section]
\newtheorem{theorem}[thm]{Theorem}
\newtheorem{lem}[thm]{Lemma}
\newtheorem{prop}[thm]{Proposition}
\newtheorem{conj}[thm]{Conjecture}
\newtheorem{question}[thm]{Question}

\theoremstyle{remark}
\newtheorem{remark}[thm]{Remark}
\theoremstyle{remark}
\newtheorem{example}[thm]{Example}
\theoremstyle{definition}
\newtheorem{defn}[thm]{Definition}
\newtheorem{definition}[thm]{Definition}

\author{
  Brianna Gambacini \\
  Department of Mathematics\\
  North Carolina State University\\
  \texttt{begambac@ncsu.edu} \\
   \And
 R.~Amzi Jeffs \\
  Department of Mathematics\\
  University of Washington\\
  \texttt{rajeffs@uw.edu} \\
   \And
 Sam Macdonald \\
  Department of Mathematics\\
  Willamette University\\
  \texttt{smmacdonald@willamette.edu} \\
   \And
  Anne Shiu \\
  Department of Mathematics\\
  Texas A\&M University \\
   \texttt{annejls@math.tamu.edu} \\
}

\begin{document}
\maketitle

\begin{abstract}
Neural codes are lists of subsets of neurons that fire together.  Of particular interest are neurons called place cells, which fire when an animal is in specific, usually convex regions in space.  A fundamental question, therefore, is to determine which neural codes arise from the regions of some collection of open convex sets or closed convex sets in Euclidean space.  This work focuses on 
how these two classes of codes -- open convex and closed convex codes -- are related.  
As a starting point, open convex codes have a desirable monotonicity property, 
namely, adding non-maximal codewords preserves open convexity; 
but here we show that 
this property fails to hold for closed convex codes.  
Additionally, while adding non-maximal codewords can only increase the open embedding dimension by 1, here we demonstrate that adding a single such codeword can increase the closed embedding dimension by an arbitrarily large amount.
Finally, we disprove a conjecture of Goldrup and Phillipson, and also present an example of a code that is neither open convex nor closed convex.
\end{abstract}


\section{Introduction}
Place cells are neurons that fire (are active) when an animal is in specific locations~\cite{Oke1}.  The resulting subsets of neurons that fire together, called a neural code, 
can be used by the brain to form a mental map of an animal's environment. Place cells were discovered by John O'Keefe in 1971, earning him a joint
(with May-Britt Moser and Edvard  Moser)
Nobel Prize in Physiology or Medicine in 2014.

The specific location where a place cell fires is called its {\em place field}, and this set 
is typically modeled by a convex set.  Thus, neural codes arising from place cells describe the regions cut out by intersecting convex sets.  This motivates the following question: Which neural codes arise from open convex sets in some Euclidean space?  (Each set is required to be open to account for the fact that place fields are full-dimensional, i.e. they have nonempty interior.)  
Many investigations into this question have been made in recent years (for instance,~\cite{cruz2019open, neural_ring, curto2017makes, giusti2014no, goldrup2014classification, meg_sam, IKR, lienkaemper2017obstructions, williams}).
Recent work such as \cite{jeffs2019embedding} shows that this question strictly generalizes the closely related topic of {\em intersection patterns} of convex sets (see~\cite{tancer-survey} for an overview). 


In this work, we consider the above question, and also, following~\cite{cruz2019open, goldrup2014classification}, the analogous question for closed convex sets.  Additionally, we ask how these two classes of codes -- open convex and closed convex codes -- are related.  Which codes are open convex but not closed convex (or vice-versa)?  Which codes are neither open convex nor closed convex? 

One starting point of our work is a recent ``monotonicity'' result of Cruz {\em et al.}~\cite{cruz2019open}: If two codes $\mathcal C$ and $\mathcal C'$, with  $\mathcal C \subset \mathcal C'$, generate the same simplicial complex, and $\mathcal C$ is open convex, then so is $\mathcal C'$ (see Proposition~\ref{prop:monotone}).  Hence, as open convexity is ``inherited'' from $\mathcal C$ to $\mathcal C'$, 
this result greatly simplifies the analysis of open convex codes.  
However, Cruz {\em et al.}\ 
did not know whether the analogous result holds for closed convexity~\cite{cruz2019open}, and here we show that, somewhat surprisingly, it does {\em not} (Theorem~\ref{thm:notmonotone}).

 
The mononotonicity result of Cruz {\em et al.} mentioned above can be paraphrased as follows:
adding non-maximal codewords to an open convex code yields another open convex code.  
The open convex realization of the larger code
may need to be in a Euclidean space of a higher dimension -- but this dimension need only increase by~1, if at all~\cite{cruz2019open}.
In contrast, we show here that for closed convex codes, this increase, even if finite, can be arbitrarily large (Theorem~\ref{thm:finiteincrease}).

We also disprove a conjecture of Goldrup and Phillipson~\cite{goldrup2014classification} concerning the relationship between open convex and closed convex codes (Theorem~\ref{thm:goldrup-phillipson}).  
Finally, we give the first example of a code on 8 neurons that has no ``local obstructions'' to (open or closed) convexity, but in fact is neither open convex nor closed convex (Theorem~\ref{thm:neither-open-nor-closed}). 

The outline of our work is as follows.  
Section~\ref{sec:background} provides relevant  definitions and prior results. 
In Section~\ref{sec:results}, we prove our main results, and then we end with a discussion in Section~\ref{sec:discussion}.

\section{Background}\label{sec:background}
In this section, we recall the definitions and prior results related to convexity of neural codes (Section~\ref{sec:convex}), simplicial complexes (Section~\ref{sec:mandatory}), 
and sunflowers of convex sets (Section~\ref{sec:sunflower}).

\subsection{Neural codes and convexity} \label{sec:convex}
In what follows, we use the notation
$[n] := \{1,2,...,n\}$.

\begin{definition}
A \textit{neural code on $n$ neurons}
is a set $\mathcal{C} \subset 2^{[n]}$.  Each $\sigma \in \mathcal C$ is a {\em codeword}, and 
$\sigma$ is a \textit{maximal codeword} of $\mathcal{C}$ if it is a maximal element of $\mathcal{C}$ with respect to inclusion.
\end{definition}

For example, the codeword $\sigma = \{1,3,4\}$ indicates that neurons $1$, $3$, and $4$ are active, while all other neurons are silent. For brevity, we will write codewords without brackets or commas; for instance,  $\sigma = 134$. Also,
 when we list the codewords of a code, all maximal codewords will be in boldface. 
\vspace{1mm}

\begin{example} \label{ex:goldrup-phillipson}
The following is a neural code on $6$ neurons, with 12 codewords: 
\begin{align} \label{eq:code-goldrup-p}
\mathcal{C} ~=~ 
    \{\mathbf{123},\mathbf{124},\mathbf{135},\mathbf{236},
    12,13,14,23,24,
    1,2,
    \emptyset\}~. 
\end{align}
\end{example}

The focus of this work is on open convex and closed convex codes (see Definition~\ref{def_clopen_convex} below).  Recall that a set $V \subset \mathbb{R}^d$ is {\em convex} if the line segment joining any two points in $V$ is contained entirely within $V$.  Also, given 
subsets $U_1,U_2,\dots, U_n$ of some $\mathbb{R}^d$ and a nonempty $\sigma\subset [n]$, we use the notation ${U}_\sigma := \bigcap_{i\in\sigma} {U}_i$.

\begin{definition}
Let $\mathcal{U} = \{U_1,U_2,\dots,U_n\}$ be a family of sets 
in a set $X \subset \mathbb{R}^d$ (we call $X$ the {\em stimulus space}).  Then $ {\rm code}(\mathcal{U},X)$ is the code on $n$ neurons given by:
    \begin{equation*} 
    \sigma \in {\rm code}(\mathcal{U},X) ~ \Longleftrightarrow ~ U_\sigma \setminus \bigcup_{j \notin \sigma} U_j \neq \emptyset~,
    \end{equation*}
where $U_{\emptyset} := X$.  
A code $\mathcal{C}$ on $n$ neurons is {\em realized} by a family of sets $\mathcal{U} = \{U_1,U_2,\dots,U_n\}$ 
in a stimulus space $X \subset \mathbb{R}^d$ if $\mathcal{C} = {\rm code}(\mathcal{U},X)$. In this case, $\mathcal U$ is called a \emph{realization} of $\mathcal C$.
\end{definition}

\begin{definition}\label{def_clopen_convex}
A code $\mathcal{C} $ on $n$ neurons 
    is \textit{open convex} (respectively, {\em closed convex}) if 
    there exists a stimulus space $X \subset \mathbb{R}^d$ (for some $d$) and a family of open (respectively, closed) convex sets $\mathcal{U} = \{U_1,U_2,\dots,U_n\}$ 
    such that
     (1) each $U_i$ is a subset of $X$, and (2)
    $\mathcal{C}= {\rm code }(\mathcal{U},X)$.
The minimum such value of $d$ is the \textit{open embedding dimension} (respectively, \textit{closed embedding dimension}) of $\mathcal C$.
\end{definition}

\begin{remark} \label{rmk:stimulus-space}
For the codes in this work, we always take the stimulus space $X$ to be $\mathbb{R}^d$ (cf.~\cite[Remark 2.19]{new-obs}). 
\end{remark}

\begin{figure}[ht] 
    \centering
    \begin{tikzpicture}[scale=0.43]
    \draw[dashed, thick] (-3,0) circle [radius=5];
    \draw[dashed, thick] (3,0) circle [radius=5];
    \draw[dashed, thick] (0,2.15) ellipse (5cm and 1.8cm);
    \draw[dashed, thick] (0,-2.5) ellipse (5cm and 1.5cm);
    \draw[dashed, thick] (2.7,2.5) circle [radius=0.8];
    \draw[dashed, thick] (-2.7,2.5) circle [radius=0.8];
    \node[] at (-6.5,0) {\Large{1}};
    \node[] at (6.5,0) {\Large{2}};
    \node[] at (-3,-2.5) {\Large{14}};
    \node[] at (3,-2.5) {\Large{24}};
    \node[] at (0,-2.5) {\Large{124}};
    \node[] at (0,-0.3) {\Large{12}};
    \node[] at (0,5) {\Large{$\emptyset$}};
    \node[] at (0,2) {\Large{123}};
    \node[] at (-2.7,2.5) {\large{135}};
    \node[] at (2.7,2.5) {\large{236}};
    \node[] at (-4,1.7) {\Large{13}};
    \node[] at (4,1.7) {\Large{23}};
    \end{tikzpicture}
    \caption{Open-convex realization of the code in Example~\ref{ex:goldrup-phillipson}.
    \label{fig:bri_counter_open}}
    \end{figure} 

\begin{remark} \label{rmk:open-dim} 
The open embedding dimension is also called the ``minimal embedding dimension''~\cite{neural_ring}.
\end{remark}
\vspace{1mm}

\begin{example}[Example~\ref{ex:goldrup-phillipson} continued] \label{ex:goldrup-phillipson-again}
Consider again the code $\mathcal C$ in~\eqref{eq:code-goldrup-p}.  
First, $\mathcal{C}$ is open convex: an open-convex realization is shown in Figure~\ref{fig:bri_counter_open} (more precisely, each set $U_i$ is the interior of the union of all closures of regions labeled by some codeword containing $i$).
Also, $\mathcal{C}$ is closed convex.  Indeed, by replacing each $U_i$ in 
Figure~\ref{fig:bri_counter_open} by its closure, we obtain a closed-convex realization of $\mathcal C$.
\end{example}
\vspace{1mm}
%
%
%
%
%


We end this subsection with two more useful definitions.

\begin{definition}\label{def_max_intersection-complete}
A code $\mathcal{C}$ is \textit{max-intersection complete} if 
every intersection of two or more maximal codewords is in $\mathcal{C}$. Otherwise, $\mathcal C$ is \textit{max-intersection incomplete}.
\end{definition}

If a code is max-intersection complete, then it is both open convex and closed convex~\cite{cruz2019open}.  The converse, however, is not true. For instance, the code $\mathcal C$ in~\eqref{eq:code-goldrup-p} is open convex and closed convex (see Example~\ref{ex:goldrup-phillipson-again}), but not max-intersection complete ($135 \cap 236=3$ is not in $\mathcal C$).
\vspace{1mm}

\begin{definition} \label{def:restricted}
Let $\mathcal{C}$ be a code on $n$ neurons, and let $\tau \subset [n]$.  The \textit{code obtained from $\mathcal{C}$ by restricting to $\tau$} is the neural code $\{\sigma \cap \tau \mid \sigma \in \mathcal{C} \}$.
\end{definition}
 Restricting to a set of neurons may be interpreted geometrically: if $\mathcal U = \{U_1,U_2, \ldots, U_n\}$ is a realization of a code $\mathcal C$, then $\{U_i\mid i\in\tau\}$ is a realization of the code obtained from $\mathcal C$ by restricting to $\tau$.

\subsection{Simplicial complexes and mandatory codewords} \label{sec:mandatory}
An (abstract) \textit{simplicial complex} on $[n]$ is a subset of $2^{[n]}$ that is closed under taking subsets. 

\begin{definition}
For a neural code $\mathcal{C}$ on $n$ neurons, the \textit{simplicial complex of} $\mathcal C$ is the smallest simplicial complex containing $\mathcal C$:
$$ \Delta(\mathcal{C}) ~:=~ \{\sigma \subset [n] : \sigma \subset \alpha \text{ for some } \alpha \in \mathcal C \}~.$$
\end{definition}

\begin{example}[Example~\ref{ex:goldrup-phillipson-again} continued] \label{ex:goldrup-phillipson-simplicial-cpx}
The simplicial complex $\Delta(\mathcal{C})$ of the code $\mathcal C$ in~\eqref{eq:code-goldrup-p} has maximal faces $123$, $124$, $135$, and $236$. The geometric realization of $\Delta(\mathcal{C})$ is shown in Figure~\ref{fig:simplicial-cpx}.
\end{example}
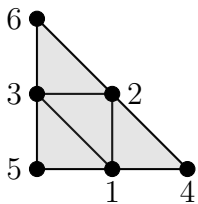
\begin{figure}[htb]
    \centering
\begin{tikzpicture}[scale=1]
    \draw [thick] (0,0) -- (0,2) -- (2,0) -- (0,0);
    \draw [thick] (1,0) -- (1,1) -- (0,1) -- (1,0);
    \draw [thick, fill=gray, opacity=0.2] (0,0) -- (0,2) -- (2,0) -- (0,0);
    \draw[fill=black] (0,0) circle [radius = 0.1] {};
    \draw[fill=black] (0,1) circle [radius = 0.1] {};
    \draw[fill=black] (0,2) circle [radius = 0.1] {};
    \draw[fill=black] (1,0) circle [radius = 0.1] {};
    \draw[fill=black] (1,1) circle [radius = 0.1] {};
    \draw[fill=black] (2,0) circle [radius = 0.1] {};
    \node[black] at (-0.3,0) {$5$};
    \node[black] at (-0.3,1) {$3$};
    \node[black] at (-0.3,2) {$6$};
    \node[black] at (1.3,1) {$2$};
    \node[black] at (1,-0.3) {$1$};
    \node[black] at (2,-0.3) {$4$};
\end{tikzpicture}
    \caption{The simplicial complex of the code in Example~\ref{ex:goldrup-phillipson}.}
    \label{fig:simplicial-cpx}
\end{figure}

The following result, due to Cruz {\em et al.}~\cite{cruz2019open}, states that for codes having the same simplicial complex, open-convexity is a monotone property with respect to inclusion:
\begin{prop}[Monotonicity property for open convex codes~\cite{cruz2019open}] \label{prop:monotone} 
Let $\mathcal C$ and $\mathcal C'$ be codes with $\mathcal{C} \subset \mathcal{C}'$ and $\Delta(\mathcal{C}) = \Delta(\mathcal{C}')$.  If $\mathcal{C}$ is open convex, then $\mathcal{C}'$ is also open convex
and, additionally, 
the open embedding dimension of $\mathcal{C}'$ is at most 1 more than that of $\mathcal{C}$.
\end{prop}


\begin{definition}\label{def_link}
Let $\Delta$ be a simplicial complex on $[n]$ and let $\sigma \in \Delta$. The \textit{link} of $\sigma$ in $\Delta$ is: $$\text{Lk}_\sigma(\Delta) := \{\tau \subset [n] \backslash \sigma : \sigma \cup \tau \in \Delta\}.$$
\end{definition}

Recall that a {\em contractible} set, by definition, is homotopy-equivalent to a single point.


\begin{definition}\label{def_mandatory}
Let $\Delta$ be a simplicial complex.  A nonempty face $\sigma \in \Delta(\mathcal{C})$ is a \textit{mandatory codeword of} $\Delta$ if (the geometric realization of) $\text{Lk}_\sigma(\Delta)$ is non-contractible. Otherwise, $\sigma$ is \textit{non-mandatory}.
\end{definition}

The following definition, pertaining to codes without certain ``local obstructions'' to convexity, is equivalent to the original definition~\cite{curto2017makes}.

\begin{definition}\label{def_local}
A code $\mathcal C$ is {\em locally good} if it contains every mandatory codeword of $\Delta(\mathcal C)$.
\end{definition}

If a code is open convex or closed convex, then it is locally good~\cite{cruz2019open, giusti2014no}.

\subsection{Sunflowers} \label{sec:sunflower}

A \emph{sunflower} is a collection of sets whose pairwise intersections are all equal and nonempty. We will be interested in sunflowers that consist of convex  sets, as introduced in \cite{jeffs2019sunflowers}. 
We define sunflowers using codes as follows.

\begin{definition}\label{def:sunflower}
A collection $\mathcal U = \{U_1, U_2, \ldots, U_n\}$ of convex sets is a \emph{sunflower} if ${\rm code}(\mathcal U, \R^d)$ contains the codeword $[n]$, and all other codewords have size at most one. When $\mathcal U$ is a sunflower, we refer to the $U_i$ as \emph{petals}. 
\end{definition}

A 3-petal sunflower is shown in Figure~\ref{fig:sunflower}.

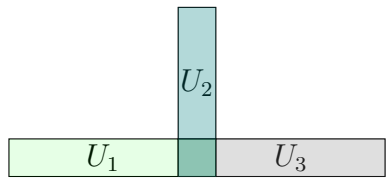
\begin{figure}[ht]
\begin{center}
\begin{tikzpicture}[scale=0.25]
\draw (0,0) to (0,2) to (2,2) to (2,0) to (0,0);
\draw (0,0) to (-9,0) to (-9,2) to (0,2) to (0,0);
\draw (2,0) to (2,2) to (11,2) to (11,0) to (2,0);
\draw (0,2) to (0,9) to (2,9) to (2,2) to (0,2);
\node[black] at (-4,1) { $U_{1}$};
\node[black] at (6,1) { $U_{3}$};
\node[black] at (1,5) { $U_{2}$};
\draw [fill=gray, opacity=0.25] (0,0) rectangle (11,2);
\draw [fill=green, opacity=0.1] (-9,0) rectangle (2,2);
\draw [fill=teal, opacity=0.3] (0,0) rectangle (2,9);
\end{tikzpicture}
\caption{A sunflower $\mathcal U = \{U_1, U_2, U_3\}$, with ${\rm code}(\mathcal U, \R^2) = \{123, 1, 2, 3, \emptyset\} $.} \label{fig:sunflower}
\end{center}
\end{figure}

For our work, we will require the following theorem which constrains how the sets in a sunflower consisting of convex open sets may be arranged.

\begin{theorem}[Sunflower Theorem \cite{jeffs2019sunflowers}]\label{thm:sunflower}
Let $\mathcal U = \{U_1, U_2, \ldots, U_n\}$ be a sunflower of convex open sets in $\R^d$, and assume that $n > d$. Then every hyperplanein $\R^d$ that has nonempty intersection with every $U_i$ also has nonempty intersection with $U_{[n]}$.
\end{theorem}

\section{Results} \label{sec:results}
Our main results are as follows.
First, closed convex codes do {\em not} possess the same monotonicity property that open convex codes have (Theorem~\ref{thm:notmonotone}).  
Next, adding non-maximal codewords can increase the embedding dimension of closed convex codes by arbitrarily large, finite amounts (Theorem~\ref{thm:finiteincrease}).
We also disprove a conjecture on the relationship between open convexity and closed convexity (Theorem~\ref{thm:goldrup-phillipson}).  
Finally, we give an example of code on 8 neurons that is locally good, but neither open convex nor closed convex (Theorem~\ref{thm:neither-open-nor-closed}), and then conjecture that there are no such codes on fewer neurons.

\subsection{Closed convexity is non-monotone} \label{sec:monotone}
Recall that open convex codes have a monontonicity property (Proposition~\ref{prop:monotone}).  It is natural to ask whether the same is true for closed convexity (indeed, Cruz {\em et al.} did not know the answer~\cite[\S 3]{cruz2019open}):
\begin{question} \label{q:closed-monotone}
Let $\mathcal C$ and $\mathcal C'$ be codes with $\mathcal{C} \subset \mathcal{C}'$ and $\Delta(\mathcal{C}) = \Delta(\mathcal{C}')$.  
\begin{itemize}
    \item [(a)] If $\mathcal{C}$ is closed convex, does it follow that $\mathcal{C}'$ is also closed convex?
    \item [(b)] If $\mathcal{C}$ and $\mathcal{C}'$ are closed convex, does it follow that the closed embedding dimension of $\mathcal{C}'$ is at most 1 more than that of $\mathcal{C}$?
\end{itemize}
\end{question}

In a special case, Question~\ref{q:closed-monotone}(a) has an affirmative answer.
Specifically, this is true for closed convex codes that have a realization in which the region of each codeword is top-dimensional (including max-intersection-complete codes); this result
follows from results of Cruz {\em et al.}~\cite[Theorem~1.3, Lemma~2.11, and Theorem 2.12]{cruz2019open}.
In general, however, Question~\ref{q:closed-monotone}(a) 
and~(b)  
have a negative answer.  We show this perhaps surprising result
 in the following theorem and Theorem~\ref{thm:finiteincrease}. 

\begin{thm}[Closed convexity is non-monotone]\label{thm:notmonotone}
Consider the code 
\[\mathcal C = 
    \{
     \mathbf{12378}, 
    \mathbf{1457}, \mathbf{2456}, \mathbf{3468},
    17, 38, 45, 46, 
    2, 
    \emptyset
    \}.\] 
This code has a closed convex realization in $\R^2$, but $\mathcal C\cup \{278\}$ is not closed convex (in any dimension).
\end{thm}

We first require a lemma regarding a closely related code, $\mathcal{C}_0$, which is (up to permutation of neurons) the minimally non-open-convex code of~\cite[Theorem 5.10]{jeffs2019morphisms} (see also~\cite[Theorem 4.2]{jeffs2019sunflowers}). 
A convex set $Y\subset \R^d$ is \emph{full-dimensional} if its affine hull is $\R^d$. Note that for a convex set, being full-dimensional is equivalent to being top-dimensional. Moreover, a convex set is full-dimensional if and only if it has nonempty interior. 

\begin{lem}\label{lem:positivecodimension}
The code $\mathcal C_0 = 
\{  \mathbf{2456},
    \mathbf{123},\mathbf{145},\mathbf{346},
     45, 46, 
    1,2,3,
    \emptyset
    \}$ 
is closed convex in $\R^2$, and every closed convex realization $\{V_1, V_2, \ldots, V_6\}$ in $\R^2$ 
is such that
$V_{123}$ is not full-dimensional.
\end{lem}
\begin{proof}
A closed convex realization of $\mathcal C_0$ in $\R^2$ is shown in Figure \ref{fig:C0realization}.

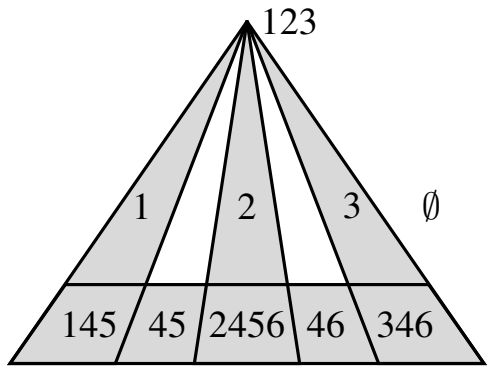
\begin{figure}[ht] 
    \centering
    \begin{tikzpicture}[scale=0.35]
    \draw[fill=black!15] (0,7)--(-9,-6)--(-5,-6)--(0,7);
    \draw[fill=black!15] (0,7)--(-2,-6)--(2,-6)--(0,7);
    \draw[fill=black!15] (0,7)--(9,-6)--(5,-6)--(0,7);
    \draw[fill=black!15] (-9,-6)--(9,-6)--(6.9,-3)--(-6.9,-3);
    \draw[very thick, line cap = round] (0,7)--(-9,-6)--(-5,-6)--(0,7);
    \draw[very thick, line cap = round] (0,7)--(-2,-6)--(2,-6)--(0,7);
    \draw[very thick, line cap = round] (0,7)--(9,-6)--(5,-6)--(0,7);
    \draw[very thick] (-9,-6)--(9,-6)--(6.9,-3)--(-6.9,-3)--cycle;
 
    \node[] at (7,0) {\Large{$\emptyset$}};
    \node[] at (-6,-4.5) {\Large{145}};
    \node[] at (-3,-4.5) {\Large{45}};
    \node[] at (0,-4.5) {\Large{2456}};
    \node[] at (3,-4.5) {\Large{46}};
    \node[] at (6,-4.5) {\Large{346}};
    \node[] at (-4, 0) {\Large{1}};
    \node[] at (0,0) {\Large{2}};
    \node[] at (4,0) {\Large{3}};
    \node[] at (1.6, 7) {\Large{123}};
    \end{tikzpicture}
    \caption{A closed convex realization of $\mathcal C_0$ in $\R^2$.}
    \label{fig:C0realization}
    \end{figure} 

To prove the rest of the lemma, let $\{V_1,V_2, \ldots, V_6\}$ be a closed convex realization of $\mathcal C_0$ in $\R^2$. 
By intersecting each of the $V_i$'s by a single sufficiently large closed ball, we may assume that each $V_i$ is compact (cf.~\cite[Remark 2.19]{new-obs}).
We will show that $V_{123}$ is not full-dimensional. 
 Below, we let $U_i$ denote the interior of $V_i$ (for $1 \leq i \leq 6$).  

Suppose for contradiction that $V_{123}$ is full-dimensional.  
Then $V_1$, $V_2$, and $V_3$ are full-dimensional.  
We claim that $\{U_1,U_2,U_3 \}$ forms a sunflower.  Indeed, $123$ is the only codeword containing more than one neuron from the set $\{1,2,3\}$, so $V_i\cap V_j = V_{123} \neq \emptyset$ for $1 \leq i < j \leq 3$.  
The same relationship holds for $\{U_1, U_2, U_3\}$ because 
$V_{123}$ is full-dimensional (and so has nonempty interior)
and 
the intersection of the interiors of two sets is the interior of their intersection. 

Next, no codeword contains $1234$, so $V_4$ is disjoint from $V_{123}$.  It follows that there exists a line $L$ properly separating the two (compact and convex) sets.  We now claim that $L$ intersects $U_i$, for $1 \leq i \leq 3$.  This claim follows from the fact that one side of $L$ properly contains a point from $V_{123}$ and 
the other side properly contains a point from the region corresponding to the codeword $145$ (or, respectively, $2456$ or $346$) which is contained in $V_4$.

In summary, $\{U_1,U_2,U_3 \}$ forms a sunflower of open, convex sets in $\mathbb{R}^2$; and the line 
$L$ passes through $U_1, U_2, U_3$, but does not pass through their intersection (as this intersection is contained in $V_{123}$).  These facts directly contradict Theorem~\ref{thm:sunflower}, and so $V_{123}$ is not full-dimensional.
%
%
%
\end{proof}

\begin{proof}[Proof of Theorem~\ref{thm:notmonotone}]
Notice that $\mathcal C$ is the same as  $\mathcal C_0$ from Lemma~\ref{lem:positivecodimension}, except that we have added neurons 7 and 8 which duplicate neurons 1 and 3, respectively. Thus the realization of $\mathcal C_0$ given in Figure~\ref{fig:C0realization} provides a closed realization of $\mathcal C$ in $\R^2$ by setting $V_7 = V_1$ and $V_8 = V_3$. Also, $\mathcal C_0$ is the restriction of $\mathcal C$ to the neurons $\{1,2,\ldots, 6\}$.

Now suppose for contradiction that there is a closed convex realization $\{V_1, V_2, \ldots, V_8\}$ of $\mathcal C \cup \{278\}$ in $\R^d$. It is straightforward to check that $d=1$ is impossible.
So, assume that $d\ge 2$. 

Let $p_1\in V_{1457}$, $p_2 \in V_{278}\setminus V_{12378}$, and $p_3 \in V_{3468}$ (so, $p_i \in V_i$ and the three points are distinct). 
Let $A$ be a 2-dimensional affine subspace of $\R^d$ containing $p_1, p_2,$ and $p_3$; and let $W_i = V_i\cap A$. We claim that $\{W_1,W_2,\ldots, W_8\}$ is a realization of $\mathcal C\cup \{278\}$ in $A$ (i.e. in $\R^2$), and moreover that $W_{123}$ is full-dimensional in this realization.

Clearly the code of $\{W_1, W_2, \ldots, W_8\}$ is contained in $\mathcal C\cup \{278\}$ since the $V_i$'s realize that code. So, we must show that every codeword from $\mathcal C\cup\{278\}$ arises inside $A$. By choice of $p_1, p_2,$ and $p_3$, $A$ contains points that realize the codewords $1457$, $3468$, and $278$. 

Consider the line segment $L_1$ from $p_2$ to $p_3$. This line segment is contained entirely in $W_8$, and so the codewords that appear along it must come from the set $\{278, 12378, 38, 3468\}$. In fact, each of these codewords must appear, and in exactly this order, since the code along the line segment must be a 1-dimensional code (see the arguments in~\cite{zvi-yan}). A symmetric argument shows that the line segment $L_2$ from $p_2$ to $p_1$ has the codewords $\{278, 12378, 17, 1457\}$ along it in that order. Finally, a similar argument shows that the line segment $L_3$ from $p_1$ to $p_3$ has along it the codewords $\{1457, 45, 2456, 46, 3468\}$ in that order.

Thus, only the codewords 
$2$ and $\emptyset$ need to be shown to arise in $A$. The codeword 2 can be recovered by examining a line segment from $p_2$ to a point in $W_{2456}$, and $\emptyset$ can be obtained by assuming that the $W_i$ are bounded (cf.\ \cite[Remark 2.19]{new-obs}).

To see that $W_{123}$ is full-dimensional in $A$, we again consider the line segments $L_1$, $L_2$, and $L_3$. The points $p_1, p_2,$ and $p_3$ must be in general position: the codeword $1457$ that $p_1$ gives rise to does not appear on the line segment $L_1$ between $p_2$ and $p_3$, the codeword $3468$ corresponding to $p_3$ does not appear on $L_2$, and the codeword $278$ corresponding to $p_2$ does not appear on $L_3$. Thus 
$p_1,p_2,p_3$ define a triangle in $\mathbb{R}^2$ with edges $L_1,L_2,L_3$; see Figure~\ref{fig:codewordtriangle}.
\begin{figure}[ht] 
    \label{fig:triangle}
    \centering
    \begin{tikzpicture}[scale=0.35]
    \draw[very thick, line cap = round] (0,7)--(-9,-6)--(9,-6)--(0,7)--cycle;    \draw[fill = red!50, opacity=0.4] (2.8, 3)--( -2.8, 3)--(9,-6)--(2.8,3)--cycle;
    \draw[fill = black!50, opacity=0.4] (2.8, 3)--( -2.8, 3)--(-9,-6)--(2.8,3)--cycle;
    \draw[fill=black] (2.8, 3) circle [radius = 0.15] {};
    \draw[fill=black] (-2.8, 3) circle [radius = 0.15] {};
    \draw[fill=black] (0,7) circle [radius = 0.15] {};
    \draw[fill=black] (-9,-6) circle [radius = 0.15] {};
    \draw[fill=black] (9,-6) circle [radius = 0.15] {};
    \node[] at (10,-6) {\Large{$p_3$}};
    \node[] at (-10,-6) {\Large{$p_1$}};
    \node[] at (0, 8) {\Large{$p_2$}};
    \node[] at (2.5, 6) {\Large{$278$}};
    \node[] at (5, 3) {\Large{$12378$}};
    \node[] at (6.5, 0) {\Large{$38$}};
    \node[] at (9, -3) {\Large{$3468$}};
    
    \node[] at (-2.5, 6) {\Large{$278$}};
    \node[] at (-5, 3) {\Large{$12378$}};
    \node[] at (-6.5, 0) {\Large{$17$}};
    \node[] at (-9, -3) {\Large{$1457$}};
    
    \node[] at (-7, -7) {\Large{$1457$}};
    \node[] at (-3.5, -7) {\Large{$45$}};
    \node[] at (0,-7) {\Large{$2456$}};
    \node[] at (3.5,-7) {\Large{$46$}};
    \node[] at (7, -7) {\Large{$3468$}};
    
    \node[] at (10, 3) {\Large{$L_1$}};
    \node[] at (-10, 3) {\Large{$L_2$}};
    \node[] at (0, -9.5) {\Large{$L_3$}};

    \node[] at (-1.5, 3.6) {\Large{$q_1$}};
    \node[] at (1.5, 3.6) {\Large{$q_2$}};

    \node[] at (-4.5,-1.5) {\Large{$T_1$}};
    \node[] at (4.5, -1.5) {\Large{$T_3$}};
    \end{tikzpicture}
    \caption{The triangle with vertices $p_1$, $p_2$, and $p_3$.  Also depicted are the codewords appearing along each edge.  The points $q_1,q_2$ and triangles $T_1,T_2$ are as in the proof of Theorem~\ref{thm:notmonotone}.}
    \label{fig:codewordtriangle}
    \end{figure}
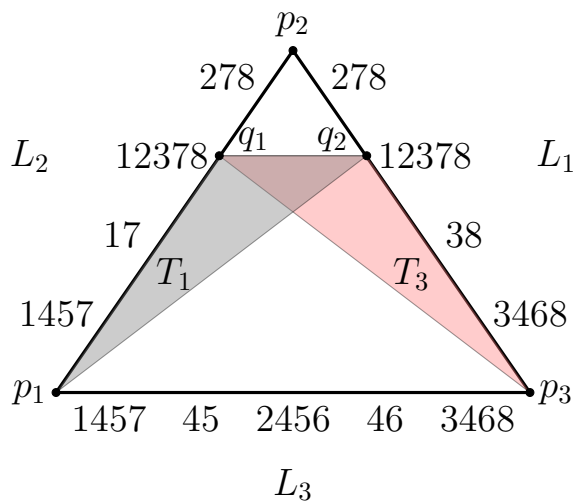 

Next, $L_1$ and $L_2$ both pass through $W_{12378}$ and intersect only at $p_2$, so we now choose distinct points $q_1$ and $q_2$ in $L_1\cap W_{12378}$ and $L_2\cap W_{12378}$, respectively. Now consider the triangles $T_1$ and $T_3$ with respective vertex sets 
$\{q_1, q_2, p_1\}$ and $\{q_1,q_2,p_3\}$ (see the figure). 
The vertices of $T_1$ are contained in $W_1$, so $T_1 \subset W_1$.  Similarly, $T_3 \subset W_3$.  Hence, $T_1 \cap T_3 \subset W_1 \cap W_3$.  
The intersection $T_1\cap T_2$ is full-dimensional (the doubly shaded region in Figure~\ref{fig:codewordtriangle}), 
and therefore so is $W_1\cap W_3$.

%

However, $W_1\cap W_3 = W_{123} = W_{12378}$ (because only the codeword $12378$ contains both neurons $1$ and $3$). So, by deleting the sets $W_7$ and $W_8$, we obtain a closed convex realization $\{W_1, W_2, \ldots, W_6\}$ of the code $\mathcal C_0$ in $A\cong\R^2$ with $W_{123}$ full-dimensional in $A$. This contradicts Lemma~\ref{lem:positivecodimension}, and so the proof is complete.
\end{proof}

\begin{remark} \label{rmk:closed-cvx-monotone}
 Theorem~\ref{thm:notmonotone} answered Question~\ref{q:closed-monotone} in the negative using a code on 8 neurons and codewords of size up to 5. We do not know whether such a result is possible using 7 or fewer neurons and/or codewords of size at most 4.
\end{remark}

\begin{remark} \label{rmk:prove-not-closed-cvx}
Previous works such as~\cite[Lemma 2.9]{cruz2019open} and~\cite[Theorem 4.1]{goldrup2014classification} have used minimum-distance arguments to prove that certain codes are not closed convex. Our proof of Theorem~\ref{thm:notmonotone} took a different approach, effectively reducing the argument to the case of open sets. 
In the future, we would like 
a general set of criteria that preclude closed convexity, and which prove, as special cases, that the 
code $\mathcal C$ of Theorem~\ref{thm:notmonotone} and the relevant codes in~\cite{cruz2019open,goldrup2014classification} are not closed convex.
\end{remark}

\subsection{Arbitrarily large increases in closed embedding dimension} \label{sec:large-increase-closed-emb-dim}

In Theorem~\ref{thm:notmonotone}, we saw that adding a non-maximal codeword to a closed convex code may yield a non-closed-convex code. 
Nevertheless, when the resulting code is closed convex, one might hope that its closed embedding dimension has not greatly increased, in line with the fact that open embedding dimension increases by at most 1 when a non-maximal codeword is added (recall Proposition~\ref{prop:monotone}). 

However, this is not the case. In fact, adding a non-maximal codeword may increase the closed embedding dimension by any amount, as we show in the next theorem.

\begin{defn}\label{def:Ad}
For $n\ge 2$, let $\mathcal A_n$ be the code whose neurons are $\{1,2,\ldots, n+1\}$ and $\{\overline{1},\overline{2},\ldots, \overline{n}\}$ and 
which consists of the following 
$2n+3$ codewords:
\begin{itemize}
    \item [(i)] The following three codewords: 
        $\{1, 2, \ldots, n, \overline{1},\overline{2},\ldots, \overline{n}\}$, $\{n+1\}$, and the empty set, 
    \item [(ii)] The codeword $\{i,\overline{i}, n+1\}$ for all $i\in[n]$, and
    \item[(iii)] The codeword $\{i, \overline{i}\}$ for all $i\in [n]$.
\end{itemize}
\end{defn}
For $i\in[n]$, note that the neuron $i$ appears in a codeword of $\mathcal A_n$ if and only if $\overline i$ appears. Thus the receptive fields of neurons $i$ and $\overline i$ will be identical in every realization of $\mathcal A_n$ (i.e., $U_i=U_{\overline{i}}$). 

\begin{thm}[Large increase in closed embedding dimension] \label{thm:finiteincrease}
For $n\ge 2$, the code $\mathcal A_n$ has a closed convex realization in $\R^2$, and 
the code 
$\mathcal A_n\cup \{\{\overline{1}, \overline{2}, \ldots, \overline{n}\}\}$ is closed convex with  closed embedding dimension equal to $n$. 
\end{thm}

To prove Theorem~\ref{thm:finiteincrease} we first require a lemma similar to Lemma~\ref{lem:positivecodimension}.

\begin{lem}\label{lem:Adcodimension}
Assume $n\geq 2$. 
Let $\mathcal S_n$ be the code obtained from $\mathcal A_n$ by restricting to the neurons $\{1,2,\ldots, n+1\}$. Let $\{V_1,V_2, \ldots, V_{n+1}\}$ be a closed convex realization of $\mathcal S_n$ in $\R^d$. If $d<n$, then the region $V_{[n]}$ is not full-dimensional. 
\end{lem}
\begin{proof}
 By intersecting each of the $V_i$'s by a single sufficiently large closed ball, we may assume that each $V_i$ is compact
(cf.~\cite[Remark 2.19]{new-obs}).
Suppose for contradiction that $V_{[n]}$ is full-dimensional. Then the sets $V_1,V_2, \ldots, V_n$ are also full-dimensional. 

Next, from Definition~\ref{def:Ad}, we see that 
$\mathcal S_n = \{  [n],
\{1,n+1\}, \{2,n+1\}, \dots, \{n,n+1\},
\{1\},\{2\},\dots,\{n+1\}, \emptyset \}$. 
So, $\{V_1, V_2, \ldots, V_n\}$ forms a sunflower, that is,
 $V_i \cap V_j = V_{[n]} \neq \emptyset $ for all $1 \leq i < j \leq n$.
We claim that the interiors $\{U_1,U_2, \ldots, U_n\}$ of the $V_i$ also form a sunflower. Indeed, for all $1\leq i < j \leq n$, we see that $U_i\cap U_j$ is the interior of $V_i\cap V_j = V_{[n]}$, which is nonempty because $V_{[n]}$ is full-dimensional. 


Next, observe that $V_{[n]}$ is disjoint from $V_{n+1}$, because $[n+1]$ is not a codeword of $\mathcal S_n$.   As both $V_{[n]}$ and $V_{n+1}$ are compact and convex, there exists a hyperplane $H$ properly separating the two sets. 
For all $i\in[n]$, the set $V_i$ intersects $V_{n+1}$ (because $\{i,n+1\} \in \mathcal S_n$) and so each side of $H$ properly contains a point in $V_i$. Thus, $H$ passes through the interior of $V_i$ for all $i\in[n]$, but does not pass through their common intersection $U_{[n]}$. When $d < n$, this 
contradicts Theorem~\ref{thm:sunflower}.
\end{proof}

We are now ready to prove the main result of this subsection.

\begin{proof}[Proof of Theorem~\ref{thm:finiteincrease}]
A closed convex realization $\{W_1,W_2, \dots, W_{n+1},~ W_{\overline{1}}, W_{\overline{2}}, \dots, W_{\overline{n}}, \} $ of $\mathcal A_n$ in $\mathbb{R}^2$ is shown in Figure~\ref{fig:Ad}: $n$ line segments meet at a common point, and an additional line segment 
crosses all the line segments away from the common point. 

\begin{figure}[ht] 
    \centering
    \begin{tikzpicture}[scale=0.35]
    \draw[very thick, line cap = round] (9,-4)--(0,4)--(-9,-4);
    \draw[very thick, line cap = round] (-3,-4)--(0,4);
    \draw[very thick, line cap = round] (-9,-2)--(9,-2);
    \node[] at (11,-2) {\Large{$W_{n+1}$}};
    \node[] at (9.2,-5) {\Large{$W_n = W_{\overline n}$}};
    \node[] at (-10, -5) {\Large{$W_1 = W_{\overline 1}$}};
    \node[] at (-3,-5) {\Large{$W_2 = W_{\overline 2}$}};
    \node[] at (3.5, -5) {\Large{$\cdots$}};

    \end{tikzpicture}
    \caption{A closed convex realization of $\mathcal A_n$ in $\R^2$.}
    \label{fig:Ad}
    \end{figure}
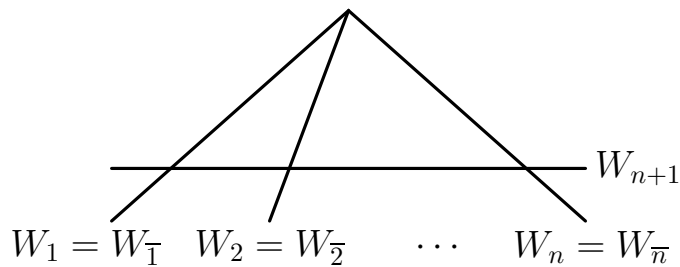 

Our next task is to construct a closed convex realization of 
$\mathcal C_n := \mathcal A_n \cup \{\{\overline{1}, \overline{2}, \ldots, \overline{n}\}\}$ in $\mathbb{R}^n$. 
An informal description of this construction is as follows.
Starting from the realization of $\mathcal{A}_n$ in Figure~\ref{fig:Ad}, rotate each 
$W_i=W_{\overline{i}}$ so it lies along the $i$-th coordinate axis in $\mathbb{R}^n$ and then fatten it into an $n$-dimensional rectangular prism, so that the common intersection is a unit $n$-cube at the origin.  Next, $W_{n+1}$ becomes a thickened hyperplane that meets each of the $W_i=W_{\overline{i}}$'s.  So far, we have a realization of $\mathcal A_n$, and now we obtain the new codeword $\{\overline{1}, \overline{2}, \ldots, \overline{n}\}$ by ``slicing off'' a corner of the $n$-cube from each of $W_1, W_2,\dots, W_n$.  This construction is shown for $n=2$ in Figure~\ref{fig:A2}.

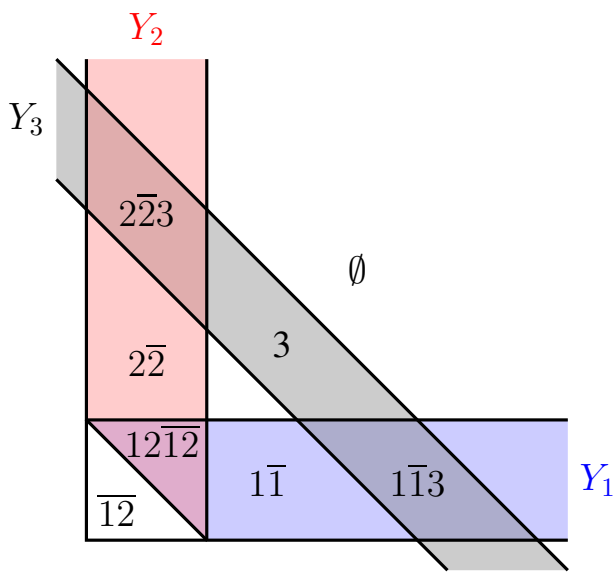
\begin{figure}[ht] 
    \centering
    \begin{tikzpicture}[scale=0.4]
    \node[] at (-9,7) {\Large{$Y_3$}};
    \node[] at (10,-5) {\Large{{\color{blue} $Y_1$}}};
    \node[] at (-5,10) {\Large{{\color{red} $Y_2$}}};
    \fill[fill = blue!50, opacity=0.4] (9,-3)--(9,-7)--(-3,-7)--(-7,-3); 
    \fill[fill = black!50, opacity=0.4] (9,-8)--(5,-8)--(-8,5)--(-8,9); 
    \fill[fill = red!50, opacity=0.4] (-7, 9)--(-3,9)--(-3,-7)--(-7,-3); 
    \draw[very thick] (9,-8)--(-8,9);
    \draw[very thick] (5,-8)--(-8,5);
    \draw[very thick] (-7, -3)--(-7,-7)--(-3,-7);
    \draw[very thick] (-3, 9)--(-3,-7);
    \draw[very thick] (-7,-3)--(9,-3);
    \draw[very thick] (9,-7)--(-3,-7)--(-7,-3)--(-7,9);

    \node[] at (-6,-6) {\Large{$\overline{12}$}};    
    \node[] at (-4.5,-3.7) {\Large{$12\overline{12}$}};
    \node[] at (-1,-5) {\Large{$1\overline{1}$}};
    \node[] at (-5,-1) {\Large{$2\overline{2}$}};
    \node[] at (4,-5) {\Large{$1\overline{1}3$}};
    \node[] at (-5,4) {\Large{$2\overline{2}3$}};
    \node[] at (-0.5,-0.5) {\Large{3}};
    \node[] at (2,2) {\Large{$\emptyset$}};
    \end{tikzpicture}
    \caption{A closed convex realization of $\mathcal{C}_2$ 
    in $\R^2$. The receptive fields $Y_1$, $Y_2$, and $Y_3$ are labeled; 
    and $Y_{\overline{1}}$ and $Y_{\overline{2}}$ are rectangles (as indicated by the codewords).}
    \label{fig:A2}
    \end{figure} 

More precisely, 
the realization 
$\{Y_1,Y_2, \dots, Y_{n+1},~ Y_{\overline{1}}, Y_{\overline{2}}, \dots , Y_{\overline{n}} \}$
of $\mathcal C_n$ in $\R^n$ is 
given by, for $i \in [n]$,
${Y}_{\overline i}:= \{x \in \mathbb{R}^n_{\geq 0} \mid 0 \leq x_j \leq 1 ~\mathrm{for~all }~ j \neq i \} $ and
$Y_i:= 
{Y}_{\overline i} \cap \{x \in \mathbb{R}^n \mid x_1+x_2 + \dots x_n \geq 1 \}$;
and 
$Y_{n+1}:=\{x \in \mathbb{R}^n \mid 2n \leq x_1+x_2 + \dots + x_n \leq 2n+1 \}$. 
%
In this realization, the codeword $\{\overline 1,\overline 2,\ldots,\overline n\}$ arises in the region of the unit $n$-cube where the sum of all coordinates is less than~1. In the remainder of the realization, we see that $i\in[n]$ appears if and only if $\overline i$ appears. Moreover, the receptive fields
$Y_i$,
for $i\in[n]$, form a sunflower whose petals meet in the subset of the $n$-cube where the sum of coordinates is greater than or equal to 1, and so the only codewords arising involving $i$ are $\{i,\overline i\}$, $\{1,2,\ldots, n,\overline 1, \overline 2,\ldots, \overline n\}$, and the codeword $\{i,\overline i, n+1\}$, the latter arising where the thickened hyperplane 
$Y_{n+1}$
meets the receptive field of $i$. Finally, the codeword $\{n+1\}$ arises anywhere in the thickened hyperplane 
$Y_{n+1}$
where it does not meet any of the $Y_i$'s,
for example, at a point where one of the coordinates is negative. 

Now it remains only to show that there is no closed convex realization of $\mathcal C_n$ in $\R^{n-1}$. Suppose for contradiction that we have such a realization 
$\{V_1,V_2, \ldots, V_{n+1}, ~ {V}_{\overline{1}},{V}_{\overline{2}}, \ldots, {V}_{\overline{n}}\}$. Choose a point $p^*$ in the region that gives rise to the codeword $\{\overline{1}, \overline{2}, \ldots, \overline{n}\}$, and for $i\in[n]$ choose a point $p_i$ in $V_i\cap V_{n+1}$ (i.e., $p_i$ lies in the region that gives rise to the codeword $\{ i, \overline{i}, n+1\}$).

For $i\in[n]$, let $L_i$ denote the line segment from $p^*$ to $p_i$. 
Observe that the containment $L_i \subset V_{\overline{i}}$ holds (as both endpoints of $L$ are in $V_{\overline{i}}$).
Also, the codewords $\{\overline{1}, \overline{2}, \ldots, \overline{n}\}$, $\{1,2,\ldots, n, \overline{1}, \overline{2}, \ldots, \overline{n}\}$, $\{i,\overline{i}\}$, and $\{i, \overline{i}, n+1\}$ appear along $L_i$ in precisely that order. 
Also, the codeword $\{n+1\}$ must arise along any line segment between distinct $p_i$. Thus, all codewords of $\mathcal C_n$ arise inside the affine hull of $\{p^*, p_1,\ldots, p_n\}$.
We denote this affine hull by $A$.

It follows that by replacing our receptive fields $\{V_1,V_2, \ldots, V_{n+1},~{V}_{\overline{1}},V_{\overline{2}}, \ldots, {V}_{\overline{n}}\}$ by their intersections with $A$, we obtain a closed convex realization of $\mathcal C_n$ inside $A\cong \R^d$, for some $d \le n-1$, such that the convex hull of the points $\{p^*,~p_1,p_2, \ldots, p_n\}$ is full-dimensional in $A$ (by construction). Observe that $d\ge 2$, as $\mathcal C_n$ is not convex in $\R^1$; one reason for this is that $\Delta(\mathcal C_n)$ has a 1-dimensional hole.

The code $\mathcal C_n$ is invariant under permutations of $[n]$ provided we also apply the permutation to $\{\overline{1},\overline{2},\ldots, \overline{n}\}$, and so we may assume without loss of generality that $\{p^*,~ p_1,p_2, \ldots, p_d\}$ form the vertices of a $d$-simplex $\Delta$ in $A$. 

For $i\in[d]$, each $L_i$ is a distinct edge of $\Delta$, and we let $q_i$ be a point on $L_i$ where the codeword $\{1,2,\ldots, n,~ \overline{1},\overline{2},\ldots, \overline{n}\}$ arises. In particular, $ p_i \neq q_i \in V_{[n]}$. Since the $q_i$ lie on distinct edges of $\Delta$, the affine hull $H$ of $\{q_1, q_2, \ldots, q_d\}$ has dimension $d-1$ and so is a hyperplane $H\subset A$. We orient $H$ so that its negative side contains $p^*$ and hence its positive side contains $\{p_1,p_2 , \ldots, p_d\}$.

For $i\in[d]$, let $\Delta_i$ be the $d$-simplex with vertices $\{q_1, q_2, \ldots, q_d,~ p_i\}$, and observe that $\Delta_i\subset V_i$. Since all $\Delta_i$ lie on the nonnegative side of $H$ and share the common face whose vertices are $\{q_1,q_2, \ldots, q_d\}$, we may choose a point $q^*$ that lies in the interior of all $\Delta_i$, and hence in $V_{[d]}$. Since $d\ge 2$ and $\{V_1, V_2, \ldots, V_n\}$ is a sunflower, we have $V_{[d]} = V_{[n]}$. Thus, $q^*$ lies in $V_{[n]}$.
The point $q^*$ lies strictly on the positive side of $H$, and so the convex hull of $\{q^*,~ q_1,q_2,\ldots, q_d\}$ is a $d$-simplex contained in $V_{[n]}$. Therefore, $V_{[n]}$ is full-dimensional in $A$. Since $d < n$, this contradicts Lemma \ref{lem:Adcodimension}.
\end{proof}

\subsection{A counterexample to a conjecture of Goldrup and Phillipson} \label{sec:Gold-Phil}
Recall that Theorem \ref{thm:notmonotone} contrasts open convex codes with closed convex codes by showing that the latter does not possess the same monotonicity property as the former. In the same spirit, 
there is much interest in comparing and relating open convexity to closed convexity. 
In particular, we would like to know properties that cause codes to be solely open convex or solely closed convex. 
In an attempt to distinguish codes that are open convex but not closed convex,
Goldrup and Phillipson posed the following conjecture~\cite[Conjecture 4.3]{goldrup2014classification}.


\begin{conj}\label{conj3} 
Let $\mathcal{C}$ be a max-intersection incomplete open convex code, where $\Delta(\mathcal{C})$ has at least two non-mandatory codewords not contained in $\mathcal C$. Suppose $\mathcal{C}$ has at least three maximal codewords $M_1, M_2, M_3$,
and there is $\sigma\subset M_1$ with $\sigma\in \mathcal{C}$ such that $\sigma\cap M_2\not\in \mathcal{C}$. Then $\mathcal{C}$ is not closed convex.
\end{conj}

We disprove Conjecture~\ref{conj3} through a counterexample, namely, the code from Example~\ref{ex:goldrup-phillipson}.

\begin{theorem} \label{thm:goldrup-phillipson}
The neural code $\mathcal{C}=
    \{
    \mathbf{123},\mathbf{124},\mathbf{135},\mathbf{236},
    12,13,14,23,24,
    1,2,
    \emptyset
    \}$ 
    fulfills the hypotheses of Conjecture \ref{conj3} and is closed convex. 
\end{theorem}

\bp 
We begin by checking that $\mathcal{C}$ satisfies the hypotheses of Conjecture 3.1. First, we saw that $\mathcal{C}$ is open convex (Example~\ref{ex:goldrup-phillipson}).  Next, $\mathcal{C}$ is max-intersection incomplete, as the intersection of maximal codewords $135 \cap 236 = 3$ is not in $\mathcal{C}$. 

We must also show that $\Delta(\mathcal{C})$ has at least two non-mandatory codewords that are not in $\mathcal C$. 
It is straightforward to check that the links $\text{Lk}_{\{3\}}(\Delta(\mathcal{C}))$ and  $\text{Lk}_{\{4\}}(\Delta(\mathcal{C}))$ are the following contractible simplicial complexes (respectively):
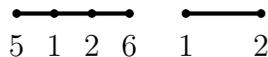
\begin{center}
    \begin{tikzpicture}[scale=0.25]
    \draw[fill=black] (-3,0) circle [radius = 0.2] {};
    \draw[fill=black] (-1,0) circle [radius = 0.2] {};
    \draw[fill=black] (1,0) circle [radius = 0.2] {};
    \draw[fill=black] (3,0) circle [radius = 0.2] {};
    \draw[ultra thick] (-3,0) -- (3,0);
    \node[] at (-3,0) [label=below:{{$5$}}] {};
    \node[] at (-1,0) [label=below:{{$1$}}] {};
    \node[] at (1,0) [label=below:{{$2$}}] {};
    \node[] at (3,0) [label=below:{{$6$}}] {};
    \draw[fill=black] (6,0) circle [radius = 0.2] {};
    \draw[fill=black] (10,0) circle [radius = 0.2] {};
    \draw[ultra thick] (6,0) -- (10,0);
    \node[] at (6,0) [label=below:{{$1$}}] {};
    \node[] at (10,0) [label=below:{{$2$}}] {};
    \end{tikzpicture}
\end{center}
Therefore, $3$ and $4$ are non-mandatory codewords.  Also, neither $3$ nor $4$ is in $\mathcal{C}$.

Next, we must show that $\mathcal{C}$ has three maximal codewords $M_1, M_2, M_3$ and a codeword $\sigma \in \mathcal{C}$ such that $\sigma \subset M_1$ and $\sigma \cap M_2 \notin \mathcal{C}$. Let $M_1 = 123, M_2 = 236, M_3 = 124$, and let $\sigma = 13 \in \mathcal{C}$. 
Then  $13=\sigma \subset M_1=123$. Also, $13 \cap 236 = \sigma \cap M_2 = 3 \notin \mathcal{C}$.

Finally, we already saw (in Example~\ref{ex:goldrup-phillipson}) that $\mathcal C$ is closed convex.  
\ep

Although Conjecture~\ref{conj3} is false, we nevertheless are interested, as mentioned earlier, in future conditions guaranteeing that a code is open convex but not closed convex -- or vice-versa.

\subsection{A locally good code that is neither open convex nor closed convex} \label{sec:neither-open-nor-closed}

Next, we consider codes that are neither open convex nor closed convex. 
Based on the examples in this work and in other articles, one might wonder whether every locally good codes is open convex or closed convex.  
 However, here
we present a code on 8 neurons that, despite being locally good, is neither open convex nor closed convex (Theorem~\ref{thm:neither-open-nor-closed}).  As seen in the proof, this code is built by combining two locally good codes, 
one that is not closed convex 
and the other not open convex. 

\begin{theorem} \label{thm:neither-open-nor-closed}
The following code 
is locally good, but neither open convex nor closed convex:
\[
\mathcal{C}~=~
 \{
 \mathbf{2345,123,124,145},
 12,14,23,24,45,
 2, 4, 
 \emptyset
 \} 
 \cup
 \{
 \mathbf{ 237,238,367,678},
 26,37,67,
 6,8
 \}~.
\] 
\end{theorem}

\begin{proof}
By~\cite[Theorem 1.3 and Lemma 1.4]{curto2017makes}, 
being locally good is equivalent to the following: If $\emptyset \neq \sigma \notin \mathcal C$ and $\sigma$ is the intersection of two or more maximal codewords of $\mathcal C$, then ${\rm Lk}_{\sigma}(\Delta{(\mathcal C}))$ is contractible.  
It is straightforward to check that only 
$1$, $3$, and $7$
are nonempty intersections of maximal codewords and not in $\mathcal C$.  Their links in $\Delta(\mathcal C)$, respectively, are shown below:
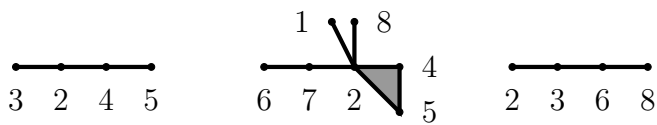
\begin{center}
    \begin{tikzpicture}[scale=0.3]
    \draw[fill=black] (-3,0) circle [radius = 0.15] {};
    \draw[fill=black] (-1,0) circle [radius = 0.15] {};
    \draw[fill=black] (1,0) circle [radius = 0.15] {};
    \draw[fill=black] (3,0) circle [radius = 0.15] {};
    \draw[ultra thick] (-3,0) -- (3,0);
    \node[] at (-3,0) [label=below:{{$3$}}] {};
    \node[] at (-1,0) [label=below:{{$2$}}] {};
    \node[] at (1,0) [label=below:{{$4$}}] {};
    \node[] at (3,0) [label=below:{{$5$}}] {};
    \draw[fill=black] (8,0) circle [radius = 0.15] {};
    \draw[fill=black] (10,0) circle [radius = 0.15] {};
    \draw[fill=black] (12,0) circle [radius = 0.15] {};
    \draw[fill=black] (11,2) circle [radius = 0.15] {};
    \draw[fill=black] (12,2) circle [radius = 0.15] {};
    \draw[fill=black] (14,0) circle [radius = 0.15] {};
    \draw[fill=black] (14,-2) circle [radius = 0.15] {};
    \fill[black!40] (14,0) -- (14, -2) -- (12,0);
    \draw[ultra thick] (8,0) -- (14,0);
    \draw[ultra thick] (12,2) -- (12,0);
    \draw[ultra thick] (11,2) -- (12,0);
    \draw[ultra thick] (14,0) -- (14, -2) -- (12,0);
    \node[] at (8,0) [label=below:{{$6$}}] {};
    \node[] at (10,0) [label=below:{{$7$}}] {};
    \node[] at (12,0) [label=below:{{$2$}}] {};
    \node[] at (12,2) [label=right:{{$8$}}] {};
    \node[] at (11,2) [label=left:{{$1$}}] {};
    \node[] at (14,0) [label=right:{{$4$}}] {};
    \node[] at (14,-2) [label=right:{{$5$}}] {};
    \draw[fill=black] (19,0) circle [radius = 0.15] {};
    \draw[fill=black] (21,0) circle [radius = 0.15] {};
    \draw[fill=black] (23,0) circle [radius = 0.15] {};
    \draw[fill=black] (25,0) circle [radius = 0.15] {};
    \draw[ultra thick] (19,0) -- (25,0);
    \node[] at (19,0) [label=below:{{$2$}}] {};
    \node[] at (21,0) [label=below:{{$3$}}] {};
    \node[] at (23,0) [label=below:{{$6$}}] {};
    \node[] at (25,0) [label=below:{{$8$}}] {};
    \end{tikzpicture}
\end{center}
Each link is contractible, and so $\mathcal C$ is locally good.

Next, we show that $\mathcal C$ is neither open convex nor closed convex.  Let $\mathcal U = \{U_1, U_2, \dots, U_8 \}$ be a realization of $\mathcal C$ in some $\mathbb{R}^d$.  
We must show that some $U_i$ is not open, and some $U_j$ is not closed.

First, $ \{U_1,U_2, U_3, U_4, U_5\}$ is a realization of 
the restriction of $\mathcal C$ to the neurons $\{1,2,3,4,5\}$, which is the code called ``$C4$'' in~\cite[Table 1]{goldrup2014classification} and is non-open-convex~\cite{lienkaemper2017obstructions}.  Thus, at least one of $U_1,U_2, U_3, U_4, U_5$ is not open.

Similarly, $ \{U_2, U_3, U_6, U_7, U_8\}$
realizes $\mathcal C$ restricted to $\{2,3,6,7,8\}$.  After relabeling neurons 
$2,3,6,7,8$ by $1,3,2,4,5$, respectively, this restricted code is the code 
 ``$C10$'' in~\cite[Table 1]{goldrup2014classification}, which is non-closed-convex~\cite[Theorem 4.1]{goldrup2014classification}.
Hence, at least one of $U_2, U_3, U_6, U_7, U_8$ is not closed. 
\end{proof}

\begin{remark} \label{rem:another-code-neither-open-nor-closed}
Another locally good code on 8 neurons that is neither open convex nor closed convex, is the code $\mathcal{C} \cup \{278\}$ from Theorem~\ref{thm:notmonotone}.
Non-closed-convexity is shown in that theorem.  As non-open-convexity, restricting the code to $\{1,2,3,4,5,6\}$ yields (up to permuting neurons) the minimally non-open-convex code in~\cite[Theorem 5.10]{jeffs2019morphisms} (this is the code in Lemma~\ref{lem:positivecodimension}), and restriction preserves convexity.  
\end{remark}

The code in Theorem~\ref{thm:neither-open-nor-closed} is on 8 neurons.  
We want to know whether there is a code with the same properties but on fewer neurons. 
(For instance, to our knowledge, the codes that Jeffs and Novik show are locally good -- in fact, ``locally perfect'' -- but neither open convex nor closed convex, require at least 8 neurons~\cite[\S 9]{Jeffs2018convex}.)

\begin{conj} \label{conj:open-closed}
Every locally good code on at most 7 neurons is open convex or closed convex.
\end{conj}
For codes on up to 4 neurons, Conjecture~\ref{conj:open-closed} is true, as such locally good codes are open convex~\cite{curto2017makes}.
For codes on 5 neurons, most of the work toward resolving the conjecture was done by Goldrup and Phillipson~\cite{goldrup2014classification}. 
The only codes left to analyze are those with the same simplicial complex as the code in~\cite[Theorem 3.1]{lienkaemper2017obstructions} (this is the code ``C4'' in the proof of Theorem~\ref{thm:neither-open-nor-closed}). 

Another approach to resolving Conjecture~\ref{conj:open-closed} comes from the fact that the set of all neural codes forms a partially ordered set (poset); for details, see~\cite{jeffs2019morphisms}.  In this poset, the open convex codes form a down-set, that is, all codes lying below an open convex code are also open convex.  Also forming down-sets are closed-convex codes~\cite[Proposition 9.3]{jeffs2019embedding} and locally good codes~\cite[Corollary 4.2]{kunin2020oriented}.  Therefore, it would be interesting to check whether any codes lying below the 8-neuron code in Theorem~\ref{thm:neither-open-nor-closed} are neither open convex nor closed convex (i.e., whether or not this code is minimal among codes that are neither open nor closed convex). If trying this approach, however, one should beware that it is possible for the number of neurons in a code to \emph{increase} while moving downwards in this poset.

%
%
%
%
%
%
%
%

\section{Discussion} \label{sec:discussion}
Open convex and closed convex codes share several important properties.  For instance, both classes of codes are locally good (and, in fact, ``locally perfect''~\cite{jeffs2019sparse}).  Also, max-intersection complete codes are both open convex and closed convex~\cite{cruz2019open}.  However, while open convex codes possess a monotonicity property, which greatly simplifies the analysis of all codes with a given simplicial complex, here we showed that this property fails for closed convex codes (Theorem~\ref{thm:notmonotone}).  Also, even when monotonicity holds, the embedding dimension can greatly increase (Theorem~\ref{thm:finiteincrease}).

Additional results in our work also address fundamental questions pertaining to open convex and closed convex codes.  For instance, there is a locally good code on 8 neurons that is neither open convex nor closed convex (Theorem~\ref{thm:neither-open-nor-closed}). 

Our results lead to several open questions.  
First, is there an instance of non-monotonicity in codes on up to 7 neurons and/or codewords with size up to 5 (Remark~\ref{rmk:closed-cvx-monotone})?  
Next, is 
there a locally good code on 7 neurons that is neither open convex nor closed convex (Conjecture~\ref{conj:open-closed})?  
Also, are there general criteria (beyond local obstructions) for ruling out closed convexity (Remark~\ref{rmk:prove-not-closed-cvx})?  This is an important future direction, as existing approaches are somewhat ad-hoc, and progress here will therefore aid in classifying closed convex codes.

Answers to these questions, together with the results we already have on convex codes, will clarify the theories of open convex and closed convex codes.  
Specifically, we will better understand when the convexity properties are the same (for instance, for nondegenerate codes~\cite{cruz2019open}) and when they differ (as seen in this work).  
In turn, 
this knowledge contributes to answering the questions from neuroscience that originally motivated our work.  Specifically, we will better understand what types of neural codes allow the brain to represent structured environments.

\subsection*{Acknowledgements}
The authors thank Aaron Chen, Chad Giusti, 
Kaitlyn Phillipson, and Thomas Yahl for insightful comments and suggestions, and thank Nida Obatake for editorial suggestions.  The authors are also grateful to four referees for detailed comments which improved this work.
BG, SM, and AS are grateful to Ramona Heuing and
Martina Juhnke-Kubitzke for finding a crucial error in a prior version of this work.
BG and SM initiated this research in the 2019 REU in the Department of Mathematics at Texas A\&M University, supported by the NSF (DMS-1757872).  AS was supported by the NSF (DMS-1752672).  RAJ was supported by the NSF (DGE-1761124).  

\bibliographystyle{unsrt}  
\bibliography{references} 

\end{document}